\documentclass[psamsfonts]{mcom-l}

\usepackage{amssymb}

\usepackage{graphicx}

\usepackage{esint}  
\usepackage{stmaryrd} 

\newtheorem{theorem}{Theorem}[section]
\newtheorem{lemma}[theorem]{Lemma}
\theoremstyle{definition}

\theoremstyle{remark}
\newtheorem{remark}[theorem]{Remark}

\numberwithin{equation}{section}

\renewcommand{\theequation}{\thesection.\arabic{equation}}
\def\@eqnnum{{\reset@font\rm (\theequation)}}

\def\a{\alpha}

\renewcommand\O{\Omega}

\newcommand{\btau}{\mbox{\boldmath$\tau$}}

\def\bn{{\bf n}}

\def\cE{{\mathcal E}}

\def\cI{{\mathcal I}}

\def\cN{{\mathcal N}}

\def\cR{{\mathcal R}}

\def\cT{{\mathcal T}}
\def\cU{{\mathcal U}}

\def\f12{\frac12}
\def\dfrac{\displaystyle\frac}

\newcommand{\grad}{\nabla}

\newcommand{\gradt}{\nabla\cdot}

\newcommand{\tri}{|\!|\!|}

\newcommand{\bdm}{\begin{displaymath}}
\newcommand{\edm}{\end{displaymath}}
\newcommand{\beq}{\begin{equation}}
\newcommand{\eeq}{\end{equation}}
\newcommand{\beqa}{\begin{eqnarray}}
\newcommand{\eeqa}{\end{eqnarray}}
\newcommand{\beqas}{\begin{eqnarray*}}
\newcommand{\eeqas}{\end{eqnarray*}}

\begin{document}

\title[Residual-based a Posteriori Estimate for Nonconforming Elements]
{Residual-based a Posteriori Error Estimate  for Interface Problems: \\ Nonconforming Linear Elements}


\author{Zhiqiang Cai}
\address{Department of Mathematics, Purdue University \\
	      150 N. University Street \\ 
	      West Lafayette, IN 47907-2067}
\email{caiz@purdue.edu}
\thanks{This work was supported in part by the National Science Foundation
under grants DMS-1217081 and DMS-1522707, the Purdue Research Foundation,
and the Research Grants Council of the Hong Kong
SAR, China under the GRF Project No. 11303914, CityU 9042090.}

\author{Cuiyu He}
\address{Department of Mathematics, Purdue University \\ 
	      150 N. University Street \\
	      West Lafayette, IN 47907-2067}
\email{he75@purdue.edu}
\thanks{}

\author{Shun Zhang}
\address{Department of Mathematics \\
	       City University of Hong Kong\\ Hong Kong}
\email{shun.zhang@cityu.edu.hk}
\thanks{}


\date{}                


\begin{abstract}
In this paper, we study a modified residual-based a posteriori error estimator for the 
nonconforming linear finite element approximation to the interface problem. 
The reliability of the estimator is analyzed by 
a new and direct approach without using the Helmholtz decomposition.
It is proved that the estimator is reliable with constant independent of the jump of
diffusion coefficients across
the interfaces, without the assumption that the diffusion coefficient is quasi-monotone. 
Numerical results for one test problem with intersecting interfaces are also presented.
\end{abstract}

\maketitle

\section{Introduction} \label{intro}
\setcounter{equation}{0}

During the past decade, the construction, analysis, and implementation of robust a posteriori 
error estimators for various finite element approximations to partial differential equations with 
parameters have been one of the focuses of research in the field of the a posteriori error 
estimation. For the elliptic interface problem, various robust estimators have been constructed, 
analyzed, and implemented (see, e.g., \cite{BeVe:2000, Pe:2002, LuWo:04, CaZh:09, CaZh:2010, CaZh:10c,Voh:11, CaZh:11} for  conforming elements,
\cite{Ai:2005, Kim:07, CaZh:10a} for nonconforming elements, 
\cite{ CaZh:10a} for mixed elements, and
\cite{CaYeZh:2011} for discontinuous elements).
The robustness for residual based estimators in the reliability bound is established theoretically under the assumption of the 
quasi-monotone distribution of the diffusion coefficients,
see \cite{BeVe:2000} for more details. However, numerical results by many researchers including ours strongly suggest that those estimators are robust 
even when the diffusion coefficients are not 
quasi-monotone. 
In this paper, we provide a theoretical evidence for the nonconforming linear element
without the quasi-monotone assumption.

One of the key steps in obtaining the robust reliability bound of classical residual 
based estimator is to construct a modified Cl\'ement-type 
interpolation operator satisfying specific approximation and stability properties 
in the energy norm (see  \cite{BeVe:2000} for details). For the conforming linear element, the degrees of freedom are the nodal values at vertices of triangles. 
The nodal value of the modified Cl\'ement-type interpolation is defined by the average value 
of the function over connected elements whose corresponding diffusion coefficients are 
the greatest. Under the quasi-monotone assumption, Bernardi and Verf\"urth \cite{BeVe:2000} 
were able to establish the required properties of the interpolation operator to guarantee the robust reliability bound. 
A key advantage for the nonconforming linear element is 
that its degrees of freedom are nodal values at the middle points of edges
of triangles and that each middle point is shared by at most two triangles. Hence, we are able to construct a
modified Cl\'ement-type interpolation satisfying the desired properties without the quasi-monotonicity
assumption (see Section~4). 

The a posteriori error estimation 
for the nonconforming elements has been studied by many researchers. Due to the lack of the error equation, 
Dari, Duran, Padra, and Vampa \cite{DaDuPaVa:1996} established the reliability bound of 
the residual-based error estimator for the Poisson equation through the Helmholtz decomposition
of the true error. Their analysis is widely used by other researchers 
(see, e.g.,  \cite{DaDuPaVa:1995, CaBaJa:2002,Ai:2005,CaHuOr:2007}),
and the Helmholtz decomposition becomes a necessary tool for obtaining 
the reliability bound for the nonconforming elements. 
This approach has also been applied to the mixed finite
element method \cite{LoSt:2006} and discontinuous Garlerkin finite element method \cite{BeHaLa:2003, Ai:2007,CaYeZh:2011}.
It is obvious that application of their analysis to the interface problem will lead to the same distribution assumption as the conforming elements in \cite{BeVe:2000}.

Ainsworth \cite{Ai:2005} constructed  an equilibrated estimator 
without using the Cl\'ement type interpolation but the error bounds 
depend on the jump of diffusion constants.
Despite the main trend of using Helmholtz decomposition in the nonconforming finite element analysis,
there are several other interesting papers that approached differently. 
Hoppe and Wohlmuth \cite{HoWo:1996} constructed two a posteriori error estimators by using the hierarchical basis under the saturation assumption.
Schieweck \cite{Sc:2002} constructed a 
two-sided bound of the energy error using the analysis of conforming case with some simple additional arguments. 
Nevertheless, conforming Cl\'ement type interpolation was applied in that paper hence again impose the 
assumption of quasi-monotonicity.      

The purpose of this paper is to present a new and direct analysis,
which does not involve the Helmholtz decomposition, 
for estimating the reliability bound with the aim of removing the 
quasi-monotone assumption. 
To do so, our analysis makes use of (a) our newly developed error 
equation for the nonconforming finite element approximation in \cite{CaYeZh:2011}
and (b) the structure of the nonconforming elements. 
Combining with our observation on the 
modified Cl\'ement-type interpolation for the nonconforming elements, 
we are able to bound
both the element residuals and the numerical flux jumps uniformly without 
the quasi-monotonicity assumption. 
Unfortunately, we are unable to do the same for the numerical edge solution jump. 
As an alternative, we modify the edge solution jump at elements where
the quasi-monotonicity assumption is not satisfied. The modified estimator is proved to be reliable 
with constant independent of the jump of the diffusion coefficients across interfaces without the quasi-monotonicity
assumption. 
By using the standard argument (see, e.g.,  \cite{BeVe:2000}),
we also establish local efficiency bounds 
uniformly with respect to the jump of the diffusion coefficient.
This robustness is obtained for the standard (modified) indicators
without (with) the quasi-monotonicity assumption.
Nevertheless, 
numerical results presented in Section 7 for a benchmark test problem seems
to suggest that the modified indicator 
generates a better mesh than the standard indicator.

Existing residual based estimators consist of the element residual, the edge flux jump, and the 
edge tangential derivative jump due to the Helmholtz decomposition. As a by-product of 
our direct approach (see (\ref{error})), the residual based estimators to be studied in this paper replace 
the edge tangential derivative jump by the edge solution jump. 
Even though they are equivalent
in two dimensions, numerical result shows that our estimator is more accurate than
the existing estimators (see Figure 6).

The outline of the paper is as follows. The interface problem and its nonconforming finite element approximation are introduced in Section 2 as well as the $L^2$ representation
of the true error in the (broken) energy norm. The ``standard'' and modified indicators and estimators 
are presented
in Section 3. The modified Cl\'ement-type interpolation operator is defined and its approximation
properties are proved in Section~4. Robust local efficiency and global reliability bounds
are established in Sections~5 and 6, respectively. Finally, we provide some numerical results
in Section~7.

\section{Nonconforming Linear Finite Element Approximation to Interface Problem}
\setcounter{equation}{0}

\subsection{Interface Problem}
For simplicity of the presentation, we consider only two dimensions.
Extension of the results in this paper to three dimensions is straightforward.
Let $\Omega$ be a bounded, open, connected subset of $\Re^2$  
with a Lipschitz continuous boundary $\partial\Omega$.
Denote by $\bn=(n_1,\,n_2)^t$ the outward unit vector normal to
the boundary. We partition the boundary of the domain $\Omega$ into
two open subsets $\Gamma_D$ and $\Gamma_N$ such that
$\partial\Omega=\bar{\Gamma}_D\cup\bar{\Gamma}_N$ and that $\Gamma_D\cap
\Gamma_N=\emptyset$. For simplicity, we assume that $\Gamma_D$ is
not empty (i.e., $\mbox{mes}\,(\Gamma_D)\not= 0$). 
Consider the following elliptic interface problem
 \begin{equation}\label{problem}
 -\nabla\cdot \,(\alpha(x)\nabla\, u) =f
 \quad \mbox{in} \,\,\Omega
 \end{equation}
with boundary conditions
 \beq\label{bc1}
 u =g_{_D}\quad \mbox{on } \,\Gamma_D
 \quad\mbox{and}\quad
 \bn \cdot \left(\alpha \nabla\, u\right) =g_{_N}\quad \mbox{on } \,\Gamma_N,
 \eeq
where $\nabla \cdot$ and $\nabla$ are the divergence and gradient operators, respectively;
$f$, $g_{_D}$, and $g_{_N}$ are given scalar-valued functions; and
the diffusion coefficient $\alpha>0$ is piecewise constant with respect to a partition of the 
domain $\bar \O=\bigcup\limits_{i=1}^n \bar \O_i$. Here the subdomain
$\O_i$ is open and polygonal.
The jump of the $\alpha$ across interfaces (subdomain boundaries) are possibly very large. 
For simplicity, assume that $f$, $g_{_D}$, and $g_{_N}$ are piecewise linear functions.

We use the standard notations and definitions for the Sobolev spaces $H^s(\O)$ and
$H^s(\partial\O)$ for $s \ge 0$. The standard associated inner products are denoted by
$(\cdot,\,\cdot)_{s,\O}$ and $(\cdot,\,\cdot)_{s,\partial\O}$,
and their respective norms are denoted by $\|\cdot\|_{s,\O}$
and $\|\cdot\|_{s,\partial\O}$. (We omit the subscript $\O$ from the inner product and norm designation
when there is no risk of confusion.) For $s = 0$, $H^s(\O)$ coincides with $L^2(\O)$. In this
case, the inner product and norm will be denoted by $\|\cdot\|$ and $(\cdot,\,\cdot)$,
respectively. Let 
\[
H^1_{g,D}(\O) := \{v \in H^1(\, \O \,) \, : \, v=g_{_D} \, \, \mbox{on} \,\, \Gamma_D \}.
\]
The corresponding variational formulation of problem $(\ref{problem})$-($\ref{bc1}$) is to find 
$u \in H^1_{g,D}(\O)$ such that
\begin{eqnarray}\label{garlerkin}
a(u,v)=f(v), \quad \forall \, v\in H^1_{0,D}(\O),
\end{eqnarray}
where the bilinear and linear forms are defined by 
\[
a(u, v)=(\alpha(x)\nabla u, \nabla v)_\O  \quad \mbox{and} \quad f(v)=(f, v)_\O+(g_{_N},v)_{\Gamma_N}.
\]

\subsection{Nonconforming Linear Finite Element Approximation}
Let $\cT^h$ be a triangulation of the domain $\O$. Assume that $\cT^h$ is regular; i.e., for all $K \in \cT^h$,
 there exist a positive constant $\kappa$ such that
 \[
 h_K \le \kappa \, \rho_K,
 \]
 where $h_K$ denotes the diameter of the element $K$ and $\rho_K$ the diameter of the largest circle that
 may be inscribed in $K$. Note that the assumption of the mesh regularity does not exclude highly, 
locally refined meshes. Let 
\[
 \cN^{h}=\cN^{h}_I\cup \cN^{h}_D\cup \cN^{h}_N
 \quad\mbox{and}\quad
  \cE^{h}=\cE^{h}_I\cup \cE^{h}_D\cup \cE^{h}_N,
 \]
where $\cN^{h}_I$ ($\cE^h_I$) is the set of all interior vertices (edges) in 
$ \cT^{h}$, and $\cN^{h}_D$ ($\cE^h_D$) and $\cN^{h}_N$ ($\cE^h_N$) are the respective sets of 
all vertices (edges) on $\overline\Gamma_D$ and $\Gamma_N$. For each $e\in \cE^{h}$, 
denote by $m_e$ the mid-point of the edge $e$. 
Furthermore, assume that interfaces
\[
F=\{ \partial \O_i \cup \partial \O_j \, : \, i,j=1, \cdots n\}
\]
do not cut through any element $K \in \cT^h$.

Let $P_k(K)$ be the space of polynomials of degree less than or equal to 
$k$ on the element $K$. 
Denote the conforming piecewise linear 
finite element space associated with the triangulation $\cT^h$ by 
\[
\cU^{c}=\{ v\in H^1(\O) :\, v|_K \in P_1(K), \quad \forall \, K \in \cT^h \}
\]
and its subset by
\[
 \cU^{c}_{g,D}=\{ v \in \cU^{c} : v|_{\Gamma_D}=g_{_D}\}.
\]
Denote the nonconforming piecewise linear 
finite element space, i.e., the Crouzeix-Raviart element \cite{GiRa:86}, associated with the triangulation $\cT^h$ by 
\[
\cU^{nc}=\{ v \in L^2(\O) : v|_K \in P_1(K), \quad \forall \, K\in \cT^h, \mbox{ and } v \mbox{ is continuous at }
 m_e, \forall e \in \cE_I^h\}
\]
and its subset by
\[
\cU_{g,D}^{nc}=\{ v \in \cU^{nc} \,:\, v(m_e)=g_{_D}(m_e), \quad \forall \,e \in \cE_D^h\}.
\]

Let 
 \[
 H^1(\cT^h)=\{ v \in L^2(\O) : v|_K \in H^1(K), \quad \forall \, K \in \cT^h\}.
 \]
For any $v,\,w\in H^1(\cT^h)$, denote the (broken) bilinear form by
  \[
a_h(v,\,w)
=\sum_{K \in \cT^h} ( \alpha \nabla v,\, \nabla w)_K
\]
and the (broken) energy norm by
 \[
  \tri v \tri_\O=\sqrt{a_h(v,\,v)}
=\left( \sum_{K \in \cT^h} \|\alpha^{1/2} \nabla v\|_{0,K}^2 \right )^{1/2}.
\]
The nonconforming finite element approximation is to find $u_h \in \cU_{g,D}^{nc}$ such that
\begin{eqnarray} \label{non-conforming}
a_h(u_h,v)=f(v), \quad \forall \, v \in \cU_{0,D}^{nc}.
\end{eqnarray}

\subsection{$L^2$ Representation of the Error}
For each edge $e \in \cE^{h}$, denote by $h_e$ the length of $e$; denote by $\bn_e$ 
a unit vector normal to $e$.
When $e \in \cE_D^h\, \cup \, \cE_N^h$, denote by $K_e^+$ the boundary element with the edge $e$, and assume that $\bn_e$ is the unit outward normal vector of $K_e$.
For any $e \in \cE_I^h$, let $K_e^+$ and $K_e^-$ be the two elements sharing the common edge $e$ assuming that 
  \[
  \alpha^+_e\equiv\alpha_{K_e^+} \ge \alpha_{K_e^-}\equiv \alpha^-_e,
 \] 
and that $\bn_e$ coincides with the unit outward normal vector of $K_e^+.$  Denote by $v|_e^+$ and $v|_e^-$, respectively, the traces of the double valued function $v$ over $e$ restricted on $K_e^+$ and $K_e^-$.
For any $v \in H^1(\cT^h)$,  
denote the normal flux jump over edge $e \in \cE^{h}$ by
\begin{equation*}
	\big\llbracket \alpha \nabla v \cdot \bn_e \big\rrbracket _e:=
	\left\{
	\begin{array}{lll}
		(\alpha \nabla v\cdot \mathbf{n}_e)|_e^+ 
			-
		(\alpha \nabla v \cdot \mathbf{n}_e)|_e^-,& e \in \cE_I^h,\\[2mm]
		0,& e\in \cE_D^h,\\[2mm]
		(\alpha \nabla v \cdot \mathbf{n}_e)|_e-g_{_N}, & e\in \cE_N^h, 
		\end{array}
		\right.
\end{equation*}
and the value jump over edge $e \in \cE^{h}$ by
\begin{equation*}
\llbracket v \rrbracket_e:=\left\{
\begin{array}{lll}
	v|_e^+-v|_e^-, 	&e \in \cE_I^h,\\[2mm]
	v|_e-g_{_D}, 		& e\in \cE_D^h,\\[2mm]
	0, 			& e \in \cE_N^h.
\end{array}
\right.
\end{equation*}
The arithmetic average over edge $e \in \cE^{h}$ is denoted by
\begin{equation*}
\{v\}_e:= 
	\left\{
	\begin{array}{lll}
		\dfrac{1}{2}\left(v|_e^+ + v|_e^-\right), 	
			& e\in \cE_I^h, \\[2mm]
		v|_e, 
			&e\in \cE_D^h \cup \cE_N^h.
	  \end{array}
	\right.
\end{equation*}
A simple calculation leads to the following identity:
\begin{equation}\label{jump product}
	\llbracket uv \rrbracket_e=\{u\}_e\llbracket v \rrbracket_e
	+\llbracket u \rrbracket_e\{v\}_e, \quad \forall \, e \in \cE_I^h.
\end{equation}
For any $v \in \cU_{0,D}^{nc}$, it is well known that the following orthogonality property holds
\begin{equation} \label{orthogonality}
	\int_e \llbracket v \rrbracket \,ds=0, \quad \forall \,e \in \cE_I^h \cup \cE_N^h 
		\quad \mbox{and} \quad
	\int_e v \, ds=0, \quad \forall \,e \in \cE_D^h.
\end{equation}
Let $u$ and $u_h$ be the solutions of $(\ref{garlerkin})$ and (\ref{non-conforming}), respectively. It is shown in \cite{CaYeZh:2011} that
\begin{equation} \label{bilinear}
	a_h(u,v_h)
		=f(v_h) +\sum_{e\in \cE_I^h} \int_e (\alpha \nabla u \cdot \bn_e) \, \llbracket v_h \rrbracket \,ds
		+\sum_{e \in \cE_D^h } \int_{e} (\alpha \nabla u \cdot \bn_e) \, v_h \,ds, 
		\quad \forall \, v_h \in \cU_{0,D}^{nc}.
\end{equation}
Denote the true error by 
 \[
 	E=u-u_h.
 \]
Difference of  (\ref{bilinear}) and (\ref{non-conforming}) yields the following error equation:
\begin{equation} \label{Error Equation}
a_h(E,v_h)=\sum_{e\in \cE_I^h} \int_e (\alpha \nabla u \cdot \bn_e) \, \llbracket v_h \rrbracket   \,ds
			+\sum_{e \in \cE_D^h } \int_{e} (\alpha \nabla u \cdot \bn_e) \, v_h\,ds, 
				\quad \forall \,v_h \in \cU_{0,D}^{nc}.
\end{equation}
Introducing the element residual, the numerical flux jump, and the numerical solution jump  
\[
	r_K=\big( f+\nabla\cdot (\alpha \nabla u_h) \big) \big|_K, \quad \forall \, K \in \cT^h,
\]
\[
	j_{\sigma,e}=\big \llbracket \a \nabla u_h \cdot \bn_e \big \rrbracket_e 
		\quad \mbox{and} \quad
	j_{u,e}=\llbracket u_h\rrbracket_e,  \quad \forall \, e\in \cE^{h},
\]
respectively, then the true error in the (broken) energy norm may be expressed in terms of those
quantities.

\begin{lemma} \label{lemma1}
Let $E_h \in \cU_{0,D}^{nc}$ be an interpolation of $E$, then we have the following 
$L^2$ representation of the error $E$ in the (broken) energy norm:
\begin{equation} \label{error}
	a_h(E,E)=\!\!\!\!\!\sum_{K \in \cT^h} (r_K, E-E_h)_K
		-\!\!\!\!\!\sum_{e\in \cE_I^h \cup \cE_N^h } \int_e j_{\sigma,e} \, \{E-E_h\}\,ds 
	    	-\!\!\!\!\!\sum_{e\in \cE_I^h \cup \cE_D^h} \int_e \{ \alpha \nabla E \cdot \bn_e \}\, j_{u,e} \,ds.
\end{equation}
\end{lemma}

\begin{proof}
First note that $ \{\alpha \nabla u_h \cdot \bn_e\}_e$ is a constant for every 
$e \in \cE_{h}$.  
The orthogonality in 
(\ref{orthogonality}) leads to
 \begin{equation}\label{orth2}
  	\int_e \{\alpha \nabla u_h \cdot \bn_e\} \llbracket E_h \rrbracket \,ds=0, 
 	\quad \forall \,e \in \cE_I^h
 	\quad \mbox{and} \quad
 	\int_e (\alpha \nabla u_h \cdot \bn_e)\, E_h \,ds=0, 
	\quad \forall\, e \,\in \cE_D^h.
 \end{equation}
It follows from integration by parts, (\ref{jump product}), the continuities of the normal component of
the flux $ -\alpha \nabla u$ and the solution $u$, and (\ref{orth2}) that
	\begin{eqnarray*}
	&& a_h(E,\,E-E_h)
		  = \sum_{K \in \cT^h} \left(\alpha \nabla E, \nabla(E-E_h)\right) \\[2mm]
	&=&\sum_{K \in \cT^h} (r_K, \,E-E_h)_K
		+\sum_{e\in \cE_I^h} \int_e \big\llbracket (\alpha \nabla E \cdot \bn_e)\, (E-E_h) \big\rrbracket \,ds \\[2mm]
	&&\quad+\sum_{e\in \cE_D^h \cup \cE_N^h} \int_e  (\alpha \nabla E \cdot \bn_e) \, (E-E_h)  \,ds\\[2mm]
	&=&\sum_{K \in \cT^h} (r_K, \,E-E_h)_K
		+\sum_{e\in \cE_I^h} \int_e \llbracket \alpha \nabla E \cdot \bn_e \rrbracket \, \{E-E_h\}\,ds \\[2mm]
	&&\quad +\sum_{e\in \cE_I^h} \int_e \{ \alpha \nabla E \cdot \bn_e \} \, 
		\big(\llbracket E\rrbracket - \llbracket E_h\rrbracket \big)\,ds
		+\sum_{e\in \cE_D^h \cup \cE_N^h} \int_e  (\alpha \nabla E \cdot \bn_e) \, (E-E_h) \,ds 	\\[2mm]
	&=&\!\!\!\!\!\sum_{K \in \cT^h} (r_K, \,E-E_h)_K
		-\!\!\!\!\!\sum_{e\in \cE_I^h \cup \cE_N^h} \int_e j_{\sigma,e}\, \{E-E_h\}\,ds 
		-\!\!\!\!\!\sum_{e\in \cE_I^h \cup \cE_D^h} \int_e \{\alpha \nabla E \cdot \bn_e \} \, j_{u,e}\,ds\\[2mm]
	&&\quad
		-\sum_{e\in \cE_I^h} \int_e  (\alpha \nabla u \cdot \bn_e)  \, \llbracket E_h\rrbracket \,ds
		-\sum_{e\in \cE_D^h } \int_e  (\alpha \nabla u \cdot \bn_e) \, E_h \,ds,
	\end{eqnarray*}
which, together with the error equation in (\ref{Error Equation}) with $v_h=E_h$, yields
    \begin{eqnarray*}
    	&&a_h(E,E)=a_h(E,\,E-E_h)+a_h(E,E_h)\\[2mm]
	&=&\!\!\!\!\!\sum_{K \in \cT^h} (r_K, \,E-E_h)_K
		-\!\!\!\!\!\sum_{e\in \cE_I^h \cup \cE_N^h} \int_e j_{\sigma,e}\, \{E-E_h\}\,ds 
		-\!\!\!\!\!\sum_{e\in \cE_I^h \cup \cE_D^h} \int_e \{ \alpha \nabla E \cdot \bn_e \} \, j_{u,e}\,ds.
 	\end{eqnarray*}
This completes the proof of the lemma.
\end{proof}

\section{Indicator and Estimator} \label{sec 3}
\setcounter{equation}{0}
In this section, based on the $L^2$ representation of the true error in the energy norm in Lemma~2.1, 
we first introduce the ``standard'' indicator and 
the corresponding estimator.
Our standard estimator consists of the usual element residual and edge flux jump plus
the edge solution jump that replaces
the edge tangential derivative jump of existing residual based estimators.
Since the robustness of the reliability bound of estimators was established under the quasi-monotonicity 
assumption on the distribution of the diffusion coefficient,  to avoid such an assumption
we introduce a new indicator and the corresponding estimator by modifying the edge solution jump at elements where the quasi-monotonicity assumption fails.

For any $K\in \cT^h$, denote by $\cN^h_K$ and $\cE^h_K$, respectively, the sets of three
vertices and three edges of $K$. Denote the respective indicators of element residual, edge flux jump,
and edge solution jump by
\begin{equation*}
\begin{aligned}
&\eta_{R_f,K}^2=\frac{h_K^2}{\alpha_K}\, \|r_K\|_{0,K}^2,\\[2mm]
&\eta_{J_{\sigma},K}^2=\sum_{e \in \cE_K^h \cap \cE_I^h} \frac{h_e }{2\,\alpha_{e}^+}\, \| j_{\sigma,e} \|_{0,e}^2
+\sum_{e\in \cE_K^h \cap \cE_N^h} \frac{h_e}{ \alpha_e}\, \| j_{\sigma,e}\|_{0,e}^2, \\[2mm]
 \mbox{and} \quad &\eta_{J_u,K}^2=\sum_{e \in \cE_K^h \cap \cE_I^h}    \frac{\alpha_{e}^-}{2\,h_e}\, \| j_{u,e}\|_{0,e}^2+
\sum_{e\in \cE_K^h \cap \cE_D^h}\frac{\alpha_e}{ h_e}\, \| j_{u,e}\|_{0,e}^2.
\end{aligned}
\end{equation*}
Then the standard indicator associated with $K\in \cT^h$ is defined by 
\begin{equation}\label{etaK}
\eta_K=\left(\eta_{R_f,K}^2 +  \eta_{J_{\sigma},K}^2 +  \eta_{J_u,K}^2 \right)^{1/2},
 \end{equation}
and the standard estimator by 
\begin{equation}\label{eta}
\eta=\left(\sum_{K\in\cT^h} \eta_K^2 \right)^{1/2}.
 \end{equation}

\begin{remark}
Instead of using the edge tangential derivative jump as existing 
residual-based error estimators, the indicator $\eta_K$ and   
the resulting estimator $\eta$ above employ the edge solution jump $\eta_{J_u,K}$.
In two dimensions, the edge solution jump is equivalent to the tangential derivative jump
(see (\ref{jump to tangential jump})).
Nevertheless, numerical results for a test problem show that
our estimator is more accurate than the existing estimators using the edge tangential derivative jump.
\end{remark}
By the standard argument \cite{BeVe:2000}, it is shown in Section~\ref{sec:etaK}
that the indicator $\eta_K$ is efficient uniformly
with respect to the jump of the diffusion coefficient. By using the Helmholtz decomposition and the modified
Cl\'{e}ment-type interpolation, one can also prove that the estimator $\eta$ is reliable. Moreover, the 
reliability constant is independent of the jump of $\alpha(x)$ provided that the distribution of $\alpha(x)$
is quasi-monotone \cite{Pe:2002}. In order to remove this assumption, 
we present a new analysis for estimating the reliability 
bound without using the Helmholtz decomposition. 
The analysis will make use of the structure of the 
nonconforming element in two-dimensions, and it enables us to bound
both the element residual and the numerical flux jump uniformly without 
the quasi-monotonicity. 
Unfortunately, we are unable to do the same for the 
numerical solution jump. As an alternative, we modify 
the indicator of the numerical solution jump at elements where the quasi-monotonicity is not satisfied.

To this end, 
for each vertex $z \in \cN^{h}$, denote by $\omega_z^h$ and $\cE_z^h$, respectively,
 the sets of all elements $K \in \cT^h$ and all edges $e \in \cE^{h}$ having $z$ as a common vertex. Let 
 \[
  \hat \omega_z^h=\{ K \in \omega_z^h \, :\, \alpha_K=\max \limits_{K' \in \omega_z^h}    
  \alpha_{K'} \} \subset \omega_z^h
 \]
be the set of all elements in $\omega_z^h$ such that the corresponding diffusion 
coefficients are the greatest. 
For any interface intersecting point $z\in \cN^{h}$, the vertex patch $\omega_z^h$ is called quasi-monotone (see \cite{Pe:2002})
if for each $K \in \omega_z^h$, there exists a subset $\hat w_{z,K}^h$ of $\omega_z^h$ such that the union of elements in $\hat w_{z,K}^h$ is a Lipschitz domain and that 
\begin{itemize}
	\item if $z \in \cN^{h} \setminus \cN_D^h$, then $\{K\} \cup \hat \omega_z^h \subset \hat w_{z,K}^h$ and 
	$\alpha_K \le \alpha_{K'}, \quad \forall \,K' \in \hat w_{z,K}^h$;
	\item if $z \in \cN_D^h$, then $K \in \hat w_{z,K}^h$, $ \partial (\cup_{K' \in \hat w_{z,K}^h} K') \cap \Gamma_D \neq \emptyset$ and $\alpha_K \le \alpha_{K'}, \quad \forall \, K' \in \hat w_{z,K}^h$.
\end{itemize}
  Denote by
 \[
 	\cN_M=\{z\in \cN^h :\, \omega_z^h \mbox{ is not quasi-monotone }\}  
 \]
the set of all interface intersecting points whose vertex patches are not quasi-monotone. 

For each element $K \in \cT^h$, subdivide it into four sub-triangles by connecting three mid-points 
of edges of $K$, and denote by $\cT^{h/2}$ the refined triangulation. Let 
$\cN^{h/2}$ be the sets of all vertices based on $\cT^{h/2}$.
 \[
 	\cN^{h/2}=\cN^{h/2}_I\cup \cN^{h/2}_D\cup \cN^{h/2}_N
  		\quad \mbox{and} \quad
			 \cE^{h/2}=\cE^{h/2}_I\cup \cE^{h/2}_D\cup \cE^{h/2}_N
 \]
where $\cN^{h/2}_I (\cE^{h/2}_I)$, $\cN^{h/2}_D (\cE^{h/2}_D)$, and $\cN^{h/2}_N (\cN^{h/2}_N)$ are the sets of all interior vertices (edges) of $\cT^{h/2}$,
all boundary vertices (edges) on $\overline \Gamma_D$ and $\Gamma_N$ , respectively.
 Let 
 \[
 	\cU^{h/2,c}_{g,D}=\left \{
 		v \in H^1(\O)\,:\,v|_T \in P_1(T), \quad \forall \, T\in \cT^{h/2}
	 		\mbox{ and } v|_{{\Gamma_D}}=g_{_D}
 	\right \},
 \]
which is the continuous piecewise linear finite element space associated with the triangulation $\cT^{h/2}$. 

Next, we introduce an interpolation operator, $I_{h/2} : \,  \cU^{nc}_{g,D} \to \cU^{h/2,c}_{g,D}$,
from the nonconforming piecewise linear finite element space on $\cT^h$ to the
conforming piecewise linear finite element space on $\cT^{h/2}$. For a given $v\in \cU^{nc}_{g,D}$, 
the nodal values of $I_{h/2} v \in \cU^{h/2,c}_{g,D}$ are defined as follows:
 \begin{itemize}
 \item[(i)] set 
 \[(I_{h/2} v)(z)=g_{_D}(z),\quad \forall \, z \in \cN_D^h; \] 
 \item[(ii)]  set 
  \[
  (I_{h/2} v) (m_e)= v (m_e),\quad \forall\, e\in \cE^h;
  \]
 \item[(iii)]  set 
  \[
  (I_{h/2} v) (z)= v|_{K_z} (z), \quad \forall\, z\in (\cN_I^h \cup \cN_N^h),
  \]
  where $K_z$ is chosen to be one element in $\hat \omega_z^h$.

 \end{itemize}
 For each vertex $z \in \cN^{h}$, denote by $\omega_z^{h/2}$ 
the sets of all elements $T \in \cT^{h/2}$
having $z$ as a common vertex.
For element $K\in\cT^h$ with at least one vertex in $\cN_M$,
the indicator of the numerical solution jump $\eta_{{J_u},K}$  is modified as follows:

  \begin{eqnarray*}
 	\tilde \eta_{J_u,K}^2
	&=&\sum_{ z \in \cN_K^h \setminus \cN_M} 
		 \left( \sum_{e \in \cE_z^{h/2}\cap \cE_K^{h/2} \cap \cE_I^{h/2} } \frac{\alpha^{-}_e}{4 \,h_e}\| j_{u,e}\|_{0,e}^2 
 		+ \sum_{e \in \cE_z^{h/2}\cap \cE_K^{h/2} \cap \cE_D^{h/2} } \frac{\alpha_e}{\,2h_e}\| j_{u,e}\|_{0,e}^2 
		 \right)  \\ [4mm]
	&&\quad+\sum_{z \in \cN_K^h \cap \cN_M} \frac{\alpha_K}{2h_K} \big \|I_{h/2} u_h -u_h \big \|^2_{0, \partial T_{K,z}},
\end{eqnarray*}
where $T_{K,z}=\omega_z^{h/2} \cap K$.

Then the modified indicator is defined as follows:
 \begin{equation}\label{tilde-etaK}
 \tilde{\eta}_K= \left\{\begin{array}{ll}
      \left(\eta_{R_f,K}^2 +  \eta_{J_{\sigma},K}^2 + \tilde \eta_{J_u,K}^2 \right)^{1/2},
      	& \mbox{if }\, \cN_K^h \cap \cN_M \neq \emptyset
      \\[4mm]
      \eta_K, 
      	& \mbox{otherwise.}
     \end{array}\right.
\end{equation}
The corresponding modified estimator is then given by 
 \begin{equation}\label{tilde-eta}
	\tilde \eta =\left(\sum_{K\in\cT_h } \tilde \eta_{K}^2 \right)^{1/2}.
\end{equation}
\begin{remark}
In the case that $\cN_M=\emptyset$; i.e., the distribution of the diffusion coefficient is quasi-monotone, 
then $\eta_{J_u,K}=\tilde \eta_{J_u,K}$ for all $K \in \cT^h$ and hence $\tilde{\eta}_K = \eta_K$ and $\tilde{\eta} = \eta$.
\end{remark}

\section{The Modified Cl\'ement-type Interpolation}
\setcounter{equation}{0}

In this section, following the idea in \cite{BeVe:2000, DrSaWi:1996}, we introduce the modified 
Cl\'ement-type interpolation operator for the non-conforming linear element
and establish its approximation properties. 

Denote by 
 \[
 \fint_{\omega} v \,dx=\frac{1}{\mbox{meas}(\omega)} \int_{\omega} v  \,dx
 \]
the mean value of a given function $v$ on a given measurable set $\omega$ in $\cR^2$
with positive 2-dimensional Lebesgue measure $\mbox{meas}(\omega)$. With this convention, set
$$
\pi_e(v)=\left\{
	\begin{array}{lll}
		\fint _{K_e^+} v\,dx, & \forall\, e\in \cE_I^{h},
		\vspace{2mm}\\ 
		\fint _{K_e} v\,dx, & \forall\, e\in \cE_D^{h} \cup \cE_N^h.
 	\end{array}
	\right.
$$
The modified Cl\'ement interpolation operator $\cI_h : H^1(\cT^h) \rightarrow \cU^{nc}$ 
is defined by
\begin{eqnarray} \label{interpolation}
	\cI_h(v)=\sum_{e \in \cE^{h}} (\pi_{e} v) \phi_{e},
\end{eqnarray}
where $\phi_e$ is the nodal basis function of $\cU^{nc}$ which takes value $1$ at $m_e$ and takes $0$ at mid-points of other edges.

For any $K\in\cT^h$, let $\triangle_K$ be the union of elements in $\cT^h$ sharing an edge 
with $K$. For any $e\in\cE^{h}$, let $\triangle_e$ be the union of elements in $\cT^h$ having  the common edge $e$.

\begin{lemma} \label{lemma2}
For any function $v \in H^1(\cT^h)$, 
then the modified Cl\'ement interpolation satisfies 
the following approximation properties:
 \begin{equation} \label{clement1}
	 \|v-\cI_h v\|_{0,K} 
	 \lesssim \frac{h_K}{\alpha_K^{1/2}} \left(
   	\tri v \tri_{ \triangle_K}
  	+\sum_{e\in\cE^h_K}\left(\dfrac{\alpha^-_e}{h_e}\right)^{1/2}
 	 \big\|\llbracket v\rrbracket\big\|_{0,e}\right), \quad \forall \, K\in \cT^h 
\end{equation} 
and
\begin{equation} \label{clement2}
	\big \|v|_e^+-\pi_{e} v \big \|_{0,e} 
	\lesssim  \left( \frac{h_e} {\alpha^+_e} \right)^{1/2} \tri v \tri_{0, K_e^+}, 
	\quad \forall \,e \in \cE^{h}.
\end{equation}
\end{lemma}

Here and thereafter, we use the $a \lesssim b$ notation to indicate that $a \le c\,b$ 
for a further not speficified constant $c$, which depends only on the shape regularity of $\cT^h$ 
but not on the data of the underlying problems, in particular, the jump of the diffusion coefficient. Unlike
the modified Cl\'ement interpolation for the conforming elements, 
there is an extra jump term in the approximation property in (\ref{clement1}) which is due to 
the discontinuity of the function $v$ across the edges of $K$.

\begin{proof}
 For any $K \in \cT^h$, since the nodal basis functions form a partition of the unity, 
the triangle inequality gives
\begin{eqnarray*}
	\|v- \cI_h v\|_{0,K}=\big \| \sum_{e \in \cE_K^h} \phi_{e} (v- \pi_{e} v)\big \|_{0,K} 
	\le  \sum_{e \in \cE_K^h}  \| v- \pi_{e} v \|_{0,K}.
\end{eqnarray*}
Hence, to show the validity of (\ref{clement1}), it suffices to prove that 
\begin{equation}\label{clement3}
 	 \|v- \pi_e v \|_{0,K} 
 	 \lesssim 
	\frac{h_K}{\alpha_K^{1/2}} \left(
	   \tri v \tri_{ \triangle_e}
  	+\left(\dfrac{\alpha^-_e}{h_e}\right)^{1/2}
  	\big\|\llbracket v\rrbracket\big\|_{0,e}\right),
  	\quad  \forall \, e \in \cE_K^h.
\end{equation}
Since the set $\triangle_e$ contains only two elements for all $e \in \cE_I^h$, 
it is obvious that 
$K=K_e^+ $ or $K_e^-$. 
If $K=K_e^+$, then (\ref{clement3}) is a direct consequence of the Poincar\'e inequality:
 \[
 	\|v-\pi_e v \|_{0,K}
 	=\big \|v - \fint_K v \,dx \big \|_{0,K} 
 	\lesssim h_K \alpha_K^{-1/2}\tri v\tri_K.
 \]
In the case that $K=K_e^-$, the triangle and the Poincar\'e inequalities imply 
 \begin{eqnarray*}
 	&&\|v -\pi_e v\|_{0,K} 
 		\le \big\| v-\fint_K v\,dx \big\|_{0,K}
 	 	+\big \| \fint_K v \,dx -\fint_{K_e^+} v\,dx \big\|_{0,K}\\[2mm]
 	&\lesssim& h_K \alpha_K^{-1/2} \tri v\tri_K
 		+h_K^{1/2} \big \| \left(\fint_K v \,dx -v|_e^-\right)
  		+\left(v|_e^+ -\fint_{K_e^+} v\,dx\right) 
  		-\llbracket v\rrbracket\big\|_{0,e}\\[2mm]
	\!\! &\le \!\!& \!\!h_K \alpha_K^{-1/2}  \tri v\tri_K
  		+h_K^{1/2}\!\!\left(\!\! \big \| \fint_K v \,dx -v|_e^- \big\|_{0,e}\!\! +
		\big \|\fint_{K_e^+} v\,dx-v|_e^+ \big\|_{0,e} \!\!
 		+\big\|\llbracket v\rrbracket\big\|_{0,e}\!\!\right).
\end{eqnarray*}
Next, we bound the three terms above.
It follows from the trace theorem and the Poincar\'e inequality that
 \[
  	h_K^{1/2} \big \|\fint_K v\,dx-v|_e^- \big \|_{0,e} 
	 \lesssim  \big \| \fint_K v \,dx -v \big\|_{0,K} +h_K \big | \fint_K v\,dx-v \big|_{1,K} 
	 \lesssim h_K \alpha_K^{-1/2} \tri v \tri_{K}.
\]
Similarly, we have 
 \begin{equation} \label{clement4}
 	 h_K^{1/2} \big \|\fint_{K_e^+} v\,dx-v|_e^+ \big \|_{0,e} 	 
  	\lesssim h_K \alpha_{K_e^+}^{-1/2} \tri v \tri_{K_e^+}.
 \end{equation}
Note that $\alpha^-_e=\alpha_K\le \alpha_e^+$, combining above three inequalities gives
 \[
	 \|v- \pi_e v\|_{0,K} 
 	\lesssim 
 	\frac{h_K}{\alpha_K^{1/2}} \left(
    	\tri v \tri_{ \triangle_e}
  	+\left(\dfrac{\alpha^-_e}{h_e}\right)^{1/2}
 	 \big\|\llbracket v\rrbracket\big\|_{0,e}\right),
 \]
which proves the validity of (\ref{clement3}) when $K=K_e^-$. 

When $e\in \cE^h_{D} \cup \cE^h_{N}$, 
(\ref{clement3}) is a direct consequence of the Poincar\'e inequality. This completes the proof of (\ref{clement3}) and hence (\ref{clement1}).
(\ref{clement2}) is a direct consequence of (\ref{clement4}).
This completes the proof of the lemma.
\end{proof}

\section{Local Efficiency Bound}\label{sec:etaK}
\setcounter{equation}{0}

In this section, we establish local efficiency bounds for the indicators $\eta_K$ and $\tilde{\eta}_K$ defined, respectively, in (\ref{etaK}) and (\ref{tilde-etaK}).
Both the bounds are independent of the jump of the diffusion coefficient, 
this robustness is obtained without (with) the quasi-monotonicity assumption 
(see \cite{BeVe:2000})
for the indicator $\eta_K$ ($\tilde{\eta}_K$). Nevertheless, 
numerical results presented in Section 7 for a benchmark test problem seems
to suggest that the modified indicator $\tilde{\eta}_K$
generates a better mesh than the standard indicator $\eta_K$.
By using local edge and element bubble functions, $\psi_e$ and $\psi_K$
(see \cite{Ve:1996} for their definitions and properties), it is a common practice to obtain the local efficiency bound for the residual-based a posteriori error
estimator. By properly weighting terms in the
indicator by the diffusion coefficient, one can show that the local efficiency bound is robust
without the quasi-monotonicity assumption. 
For the convenience of readers, we only sketch the proof in the following theorem.

\begin{theorem} {\em (Local Efficiency)}
Assuming that $u \in H^{1+\epsilon}(\O)$ and $u_h$ are the solutions of $(\ref{garlerkin})$ and 
$(\ref{non-conforming})$, respectively,
then the indicator $\eta_K$ 
satisfies the following local efficiency bound: 
 \begin{equation}\label{eff}
 	\eta_{K}  
	 \lesssim 
 	\tri E \tri_{\triangle_K}, \quad  \forall \;  K\in \cT^h.
 \end{equation}
\end{theorem}

\begin{proof}
For any $K\in\cT^h$, it follows from the properties of $\psi_K$, integration by parts,
and the Cauchy-Schwarz inequality that
   \begin{eqnarray*}
 	\|r_K\|_{0,K}^2 
  		&\lesssim  &\int_K (f+ \gradt(\alpha \nabla u_h)) \, r_K \psi_K\,dx
		 = \int_K \alpha \nabla(u-u_h) \cdot \nabla (r_K \psi_K) \,dx\\[2mm]
 	&\lesssim & \alpha_{K}^{1/2} \tri u-u_h \tri_K \, | r_K \psi_K |_{1,K}
		 \lesssim   \alpha_{K}^{1/2} h_K^{-1}   \tri u-u_h \tri_K \, \|r_K\|_{0,K},
 \end{eqnarray*}
which implies 
 \begin{equation} \label{element residual}
	 \|r_K\|_{0,K} 
 		\lesssim \frac{\alpha_K^{1/2}}{h_K} \, \tri u-u_h\tri_K, 
		\quad \forall \,K \in \cT^h.
\end{equation}

For any $e \in \cE_I^h$, by using the properties of $\psi_e$,  integration by parts, the Cauchy-Schwartz inequality, and (\ref{element residual}), we have
\begin{eqnarray*}
    \|  j_{\sigma,e} \|_{0,e}^2 
	&\lesssim&
	\int_e \llbracket \alpha \nabla u_h \cdot \bn_e \rrbracket \,
     	j_{\sigma,e} \psi_e \,ds
	=-\sum_{K \in \triangle_e} \int_{\partial K} (\alpha \nabla E \cdot \bn) \, j_{\sigma,e} \psi_e \,ds\\[2mm]
	&=& 
	\!\!-\!\! \sum_{K \in \triangle_e}\!\! \left(\int_K (\alpha \nabla E) 
    	\cdot \nabla (j_{\sigma,e} \psi_e)\, dx
	-\int_{K} r_K \, j_{\sigma,e} \psi_e \,dx \right)\\[2mm]
	&\lesssim&
	\left(\frac{\alpha^+_e}{h_e}\right)^{1/2} 
  	 \tri E \tri_{\triangle_e} 
	\|j_{\sigma,e}\|_{0,e}.
\end{eqnarray*}
Together with a similar bound for $e \in \cE_D^h \cup \cE_N^h$,  it implies 
 \begin{equation} \label{flux jump}
	 \|j_{\sigma,e}\|_{0,e}
	\lesssim \left(\frac{\alpha^+_e}{h_e}\right)^{1/2} 
 	\tri E \tri_{\triangle_e} , \quad
	\forall \,e \in \cE^{h}.
\end{equation}

For any $e \in \cE^{h}$, let $\bn_e=(n_1,\,n_2)^t$, then ${\btau}_e=(-n_2,\,n_1)^t$ is
the unit vector tangent to the edge $e$. Denote by 
$j_{\tau,e}=\llbracket \nabla u_h \cdot \btau_e \rrbracket_e$
the jump of the tangential derivative of the numerical solution $u_h$ along the edge $e$. 
By the continuity of $u_h$ at the midpoint $m_e$, 
we have
\begin{equation} \label{jump to tangential jump}
 	\|j_{u,e}\|_{0,e}=\frac{1}{\sqrt{12}}h_e\, \|j_{\tau,e}\|_{0,e},  \quad \forall\, e \in \cE_I^h.
 \end{equation}
For a scalar-valued function $v$, denote by 
$\grad^\perp v= \left(\dfrac{\partial v}{\partial y},\,-\dfrac{\partial v}{\partial x}\right)^t$ 
the formal adjoint operator of
the curl operator in two-dimensions.
For any $e \in \cE_I^h$, it follows from 
the properties of $\psi_e$,  integration by parts, and the Cauchy-Schwartz inequality that
 \begin{eqnarray*}
    \|j_{\tau,e}\|_{0,e}^2 
   &\lesssim &
   	\int_e \big \llbracket \nabla u_h \cdot \btau_e  \big \rrbracket 
 	 \, j_{\tau,e} \psi_e \,ds 
 	= -\sum_{K \in \triangle_e} \int_{\partial K}( \nabla E \cdot \btau )
 	 \,  j_{\tau,e} \psi_e \,ds\\[1ex]
 &=&\sum_{K \in \triangle_e}
   	\int_K \nabla E \cdot \grad^\perp \,  (j_{\tau,e} \psi_e )\,dx 
 	\lesssim  \sum_{K \in \triangle_e}  \alpha_K^{-1/2}
 	\tri E\tri_{K} \,| j_{\tau,e}   \psi_e |_{1,K} \\[2mm]
 &\lesssim & \sum_{K \in \triangle_e}  \alpha_K^{-1/2}
 	\tri E\tri_{K} \, h_e^{-1/2} \|j_{\tau,e}\|_{0,e}
 	\lesssim \left( \alpha^-_e h_e \right)^{-1/2} 
 	\tri E \tri_{\triangle_e} \|j_{\tau,e}\|_{0,e},
 \end{eqnarray*}
which, together with (\ref{jump to tangential jump}) and a similar bound for $e \in \cE_D^h \cup \cE_N^h$, yields
\begin{equation} \label{value jump}
 	\|j_{u,e}\|_{0,e}
 	\lesssim \left( \frac{h_e}{\alpha_e^-} \right)^{1/2} 
 	\tri E \tri_{\triangle_e} , \quad
	\forall \, e \in \cE^{h}.
\end{equation}

Now, the efficiency bound in (\ref{eff}) is a direct consequence of the bounds in 
(\ref{element residual}),
(\ref{flux jump}), and (\ref{value jump}).
This completes the proof of the theorem.
\end{proof}

Next, we establish the local efficiency bound for the modified indicator $\tilde{\eta}_K$. 
To this end, let $K\in\cT^h$ be an element having at least one vertex $z\in \cN_M$.
From the element $K\in \omega^h_z$ to the element $K_z\in \hat{\omega}^h_z\subset \omega^h_z$,
where $K_z$ is chosen in the definition of the operator $I_{h/2}$ in Section 3, there are 
at most two possible paths (clockwise or counter-clockwise) in $\omega^h_z$. 
Denote by $\hat{\omega}_{K,z}^{h/2}$ the union of elements of $\omega_z^{h/2}$
on one of the paths such that the maximum of the ratio between $\alpha_K$ and 
the diffusion coefficients over elements  
on that path is the smallest. 
Let
\[
 	C_{K,z} \equiv \max_{T\in \hat{\omega}^{h/2}_{K,z}}\, \dfrac{\alpha_K}{\alpha_{T}}.
 \] 

\begin{remark}
If $\omega_z^h$ is quasi-monotone, then $C_{K,z}=  1$ for all $K\in \omega^h_z$.
\end{remark}

\begin{lemma} \label{effi:tildeju}
Let $u_h$ be the solution of $(\ref{non-conforming})$, and let
$K\in\cT^h$ be an element having at least one vertex $z\in \cN_M$
and let $T_{K,z}=K \cap \omega_z^{h/2}$. Let $I_{h/2}$ be the interpolation 
operator defined in {\em Section \ref{sec 3}} with $K_z\in \hat{\omega}_{K,z}^{h/2}$ 
described above. Then we have
 \begin{equation} \label{effi:1}
 	\dfrac{\alpha_K}{h_K}\,
 	\|I_{h/2} u_h- u_h\|^2_{0,\partial T_{K,z}}
  	\le 2 \,C_{K,z}  \sum_{e\in \cE_z^{h/2} }  
  	\dfrac{\alpha_e^-}{h_e}\,\big \| \llbracket u_h \rrbracket \big\|^2_{0,e} .
 \end{equation}
\end{lemma}

\begin{proof}
Without loss of generality, we prove (\ref{effi:1}) only when 
$z\in \cN_I^h$. In the case when $z \in \cN_D^h \cup \cN_N^h$,  (\ref{effi:1})  may be proved
in a similar fashion.
To this end, assume that there are $k$ elements in $\hat{\omega}_{K,z}^{h/2}$
and denote these elements by $T_i$ ($i=1,\,...,\,k$) starting from $T_1= T_{K,z}$ along the path and ending at $T_k=K_z\cap \omega_z^{h/2}$.
Denote by $u_i$ and $\alpha_i$ the restrictions of $u_h(z)$ and $\alpha$ on $T_i$, respectively,
and by $e_i$ the edge between $T_{i-1}$ and $T_i$.
A direct calculation gives that
\[
	\dfrac{\alpha_K}{h_K} \|I_{h/2} u_h - u_h\|_{0,\partial T_{K,z}}^2 
	=\dfrac{2}{3} \alpha_1 (u_k-u_1)^2
\]
and that
\[
	\dfrac{\alpha_{e_i}^-}{h_{e_i}}\|\llbracket{u_h}\rrbracket\|_{0,e_i}^2
	=\frac{1}{3} \alpha_{e_i}^-\,\, \big|u_i-u_{i-1}\big|^2,
	 \quad  i=2,\, ...\, ,k.
\]
Together with the triangle inequality, we have
\begin{eqnarray*}
&&\dfrac{\alpha_K}{h_K} \|I_{h/2} u_h - u_h\|_{0,\partial T_{K,z}}^2
	=\dfrac{2}{3}
	\alpha_1 (u_k-u_1)^2 
	\le \dfrac{2}{3} \sum_{i=2}^{k} \dfrac{\alpha_1}{\alpha_{e_i}^-}\alpha_{e_i}^- (u_i-u_{i-1})^2
	\nonumber\\[2mm]
&\le& 2\,C_{K,z}
	 \sum_{i=2}^{k} 
 	\dfrac{\alpha_{e_i}^-}{h_{e_i}}\|\llbracket{u_h}\rrbracket\|_{0,e_i}^2
       \le 2\,C_{K,z}
	\sum_{e \in \cE_z^{h/2} \cap \cE_I^{h/2}}
	 \dfrac{\alpha_{e}^-}{h_{e}}\|\llbracket{u_h}\rrbracket\|_{0,e}^2. 
\end{eqnarray*}
This completes the proof of (\ref{effi:1}) and, hence, the lemma.
\end{proof}

\begin{theorem} {\em (Local Efficiency)}
Assuming that $u \in H^{1+\epsilon}(\O)$ and $u_h$ are the solutions of $(\ref{garlerkin})$ and 
$(\ref{non-conforming})$, respectively.
Then the modified indicator $\tilde\eta_K$ 
satisfies the following local efficiency bound: 
 \begin{equation}\label{effi:2}
	 \tilde \eta_K \lesssim  \max_{z\in \cN_K^h} C_{K,z}\,\, 
	\tri E \tri_{\triangle_K},\quad  \forall \;  K\in \cT^h.
 \end{equation} 
\end{theorem}

\begin{proof}
(\ref{effi:2}) is a direct consequence of (\ref{eff}) and (\ref{effi:1}).
\end{proof}

\section{Global Reliability Bound}
\setcounter{equation}{0}

 Let 
 \[
	 \hat{\eta}_{J_u, K}= \left(\frac{\alpha_K}{h_K} \right)^{1/2}\,\| I_{h/2} u_h -u_h\|_{\partial K}
 	\quad\mbox{and}\quad
 	\hat{\eta}_{J_u}
	 =\left(\sum \limits_{K\in \cT^h} \frac{\alpha_K}{h_K}\, \| I_{h/2} u_h -u_h\|_{\partial K}^2
	 \right)^{1/2}. 
 \]

\begin{lemma}\label{lemma3}
Let $u_h$ be the solution of $(\ref{non-conforming})$ and $I_{h/2}$ be the interpolation operator defined in 
{\em Section~\ref{sec 3}}, then the jump of the numerical solution 
has the following upper bound:
 \begin{equation}\label{bound1}
 	\sum_{e \in \cE_I^h \cup \cE_D^h} 
 	 \int_e \{ \alpha \nabla E \cdot \bn_e\} \, j_{u,e} \,ds 
 	\lesssim  \hat{\eta}_{J_u}
 	 \, \tri E\tri_{\O}.
\end{equation}
\end{lemma}

\begin{proof}
Since the integral over edge segment $e\in\partial K$ on the left-hand side of inequality (\ref{bound1}) may be only
regarded as the duality pair between $H^{\delta-1/2}(e)$ and 
$H^{1/2-\delta}(e)$ for an arbitrarily small $\delta>0$, 
we are not able to bound this
integral directly. To overcome this difficulty, we express them in terms of integrals along
the boundary of elements. 
To this end, first note that 
 \[
 	\llbracket I_{h/2} u_h\rrbracket_e=0
	 \quad\mbox{and}\quad
	 \llbracket \alpha \nabla u \cdot \bn_e\rrbracket_e=0,
	 \quad\forall\, e \in \cE_I^h.
 \]
By (\ref{jump product}) and the fact that $I_{h/2} u_h=g_{_D}$ on $\Gamma_D$, we have
 \begin{eqnarray*}
  	&& -\sum_{e \in \cE_I^h \cup \cE_D^h}
 		\int_e \{ \alpha \nabla E \cdot \bn_e \}  j_{u,e} \,ds \\[2mm]
 	&=&\sum_{e\in \cE_I^h} \int_e \{ \alpha \nabla E \cdot \bn_e \}\, 
 		\llbracket I_{h/2} u_h- u_h\rrbracket \,ds
  		+\sum_{e \in \cE_D^h } \int_{e} (\alpha \nabla E \cdot \bn_e) \, (g_{_D}-u_h) \,ds\\[2mm]
 	&=&\sum_{K\in \cT^h} \int_{\partial K}  
  		 (\alpha \nabla E \cdot \bn) \, ( I_{h/2} u_h- u_h) \,ds
  		- \sum_{e\in \cE_I^h} \int_e \llbracket \alpha \nabla E \cdot \bn_e
   		\rrbracket\, \{ I_{h/2} u_h- u_h\} \,ds\\[2mm]
	&& -\sum_{e\in \cE_D^h \cup \cE_N^h} 
  		\int_e  (\alpha \nabla E \cdot \bn_e) \,  ( I_{h/2} u_h- u_h) \,ds
		+\sum_{e \in \cE_D^h } \int_{e} (\alpha\nabla E \cdot \bn_e) \, (g_{_D}-u_h)\,ds\\[2mm]
 	&=&\sum_{K\in \cT^h} \int_{\partial K}  
   		(\alpha \nabla E \cdot \bn) \, ( I_{h/2} u_h- u_h) \,ds
		+ \sum_{e\in \cE_I^h \cup \cE_N^h} \int_e j_{\sigma,e} \{ I_{h/2} u_h- u_h\} \,ds\\[2mm]
 	&\triangleq& I_1+ I_2.
\end{eqnarray*}

The $I_1$ may be bounded above by using the definition of the dual norm, the trace theorem (see, e.g., \cite{CaYeZh:2011}), the inverse inequality, and $(\ref{element residual})$  
as follows:
 \begin{eqnarray}\nonumber
  I_1	
  &\le&
  	\sum_{K \in \cT_{h}}\big \|  \alpha\nabla E \cdot \bn
  	\big\|_{-1/2, \partial K} \big\| I_{h/2} u_h- u_h\big\|_{1/2,\partial K}\\[2mm] 
  &\lesssim& \sum_{K\in \cT_{h}}  \alpha_K^{-1/2}
  	\bigg( \|\alpha \nabla E\|_{0,K}+h_K \, \| r_K\|_{0,K} \bigg)
	 \,\hat{\eta}_{J_u, K} \nonumber\\[2mm]
  & \lesssim & \hat{\eta}_{J_u}\, \tri E\tri_{\O}.\label{6.2}
 \end{eqnarray}
To bound the $I_2$, first note that 
 \[
 	\int_{e } j_{\sigma,e}\, \llbracket I_{h/2} u_h- u_h \rrbracket \,ds=0, 
 	\quad \forall \, e\in \cE_I^h,
\]
which is a consequence of the orthogonality property in $(\ref{orthogonality})$
and the facts that $j_{\sigma,e}$ is a constant and that 
$\llbracket I_{h/2} u_h \rrbracket_e=0$ for all $e \in \cE_I^h$. Hence, 
\begin{eqnarray*} \label{6.1}
	\int_e j_{\sigma,e} \{ I_{h/2} u_h- u_h\} \,ds
	&=& \int_e j_{\sigma,e} \{ I_{h/2} u_h- u_h\} \,ds
		+\frac{1}{2}\int_e j_{\sigma,e} \llbracket I_{h/2} u_h- u_h \rrbracket \,ds\nonumber 
 		\\ [2mm]
	&=&\int_e j_{\sigma,e}  (I_{h/2} u_h- u_h|_e^+ )\,ds, \quad \forall e \in \cE_I^h.
\end{eqnarray*}
Now, it follows from the Cauchy-Schwartz inequality and (\ref{flux jump}) that
\begin{eqnarray}\nonumber
 I_2
	&=& \sum_{e\in \cE_I^h\cup  \cE_N^h} \int_e j_{\sigma,e}  (I_{h/2} u_h- u_h|_e^+ )\,ds
 		\\[2mm]\nonumber
	&\le&\left( \sum_{e \in \cE_I^h \cup \cE_N^h} \frac{\alpha^+_e}{h_e}\,
     		\|I_{h/2} u_h -u_h|_e^+ \|_{0,e}^2\right)^{1/2}
		\left(\sum_{e \in \cE_I^h \cup \cE_N^h}  \frac{h_e}{\alpha^+_e}\,
  		\|j_{\sigma,e} \|_{0,e}^2\right)^{1/2}\\[2mm] 
 	&\lesssim &\left (\sum_{e \in \cE_I^h \cup \cE_N^h}  
 		\frac{\alpha^+_e}{h_e}\|I_{h/2} u_h -u_h|_e^+ \|_{0,e}^2 \right)^{1/2} \tri E\tri_{\O}
 		\nonumber\\[2mm]
 	&\lesssim&  \hat{\eta}_{J_u}\,\tri E\tri_{\O}.\label{6.3}
\end{eqnarray}	
 (\ref{bound1}) is then a consequence of (\ref{6.2}) and (\ref{6.3}). This completes the proof of the lemma.
\end{proof}

Denote by $\tilde{\eta}_{J_u} =\left(\sum\limits_{K\in\cT^h} \tilde{\eta}_{J_u,K}^2\right)^{1/2}$ 
the part of the modified estimator associated with the solution jump. Then
\[
	 \tilde \eta_{J_u}^2=\sum_{ z \in \cN_K^h \setminus \cN_M} 
 	 \sum_{e \in \cE_z^{h/2} } \frac{\alpha^{-}_e}{2 \,h_e}\| j_{u,e}\|_{0,e}^2 +
	 \sum_{z \in \cN_K^h \cap \cN_M} \sum_{T \in \omega_z^{h/2}} 
	 	\frac{\alpha_T}{4h_T} \big \|I_{h/2} u_h -u_h \big \|^2_{0, \partial T} .
\]

\begin{lemma}\label{lemma4}
 The $\hat \eta_{J_u} $ is bounded above by the $\tilde{\eta}_{J_u}$; i.e., 
\begin{equation}\label{6.4}
	\hat \eta_{J_u} \lesssim \tilde \eta_{J_u}.
\end{equation}
\end{lemma}

\begin{proof}
Since $u_h-I_{h/2} u_h$ vanishes on all boundary edges of $w_z^{h/2}$ for all $z \in \cN^{h}$, we have
 \[
	\hat \eta_{J_u}=
	\sum_{K\in \cT^{h}}  \frac{\alpha_K}{h_K}  \| I_{h/2} u_h- u_h\|_{0,\partial K}^2
 	=\sum_{z \in \cN^{h}} \sum_{T\in w_z^{h/2}}  
	 \frac{\alpha_T}{2h_T}  \| I_{h/2} u_h- u_h\|_{0,\partial T}^2.
\]
To prove the validity of (\ref{6.4}), it suffices to show that
for all $z \in  \cN^h \setminus \cN_M$,
$$
	\frac{\alpha_T}{h_T}  \| I_{h/2} u_h- u_h\|_{0,\partial T}^2
	\lesssim 
	\sum_{e \in \cE_z^{h/2} } \dfrac{\alpha_-^e}{h_e} 
	\|\llbracket u_h \rrbracket\|_{0,e}^2, \quad \forall \,\, T \in \omega_z^{h/2}.
$$  
This may be proved in a similar fashion as (\ref{effi:1}) with the fact that $C_{K,z}=1$
for all $z \in  \cN^h \setminus \cN_M$.
This completes the proof of the lemma.
\end{proof}

\begin{theorem}{\em (Global Reliability)} \label{global reliability}
Let $u$ and $u_h$ be the solutions of $(\ref{garlerkin})$ and $(\ref{non-conforming})$, respectively. Then the estimator $\tilde \eta$ satisfies the following global reliability bound: 
\begin{equation} \label{global bound}
\tri E\tri_{\O}  
\lesssim \tilde \eta.
\end{equation}
\end{theorem}

\begin{proof}
Let $\cI_h$ be the modified Cl\'ement interpolation operator defined in Section~4.
Then (\ref{error}) with $E_h=\cI_h E$ becomes
 \begin{eqnarray} \nonumber
 a_h(E,E)
 &=&\sum_{K \in \cT^h} (r_K, E-\cI_h E)_K
 	-\sum_{e\in \cE_I^h \cup \cE_N^h } \int_e j_{\sigma,e} \, \{E-\cI_h E\}ds \\[2mm] \label{error2}
	&&\quad-\sum_{e\in \cE_I^h \cup \cE_D^h} 
 	\int_e \{ \alpha \nabla E \cdot \bn_e\}\, j_{u,e} ds\nonumber\\[2mm]
	&\triangleq& I_1+ I_2 + I_3.
\end{eqnarray}
The first term in (\ref{error2}) may be bounded by  the Cauchy-Schwartz inequality, Lemma \ref{lemma2}, and (\ref{value jump}) as follows:
 \begin{eqnarray} \label{global1}
 I_1
  	&\le &\sum_{K \in \cT^h} \| r_K\|_{0,K} \, \|E-\cI_h E\|_{0,K} \nonumber\\[2mm]
	&\lesssim&\sum_{K \in \cT^h}  \eta_{R_f,K}\left(\tri E \tri_{0,\triangle_K} 
		 +\sum_{e\in\cE_K^h} \left(\dfrac{\a_e^-}{h_e}\right)^{1/2}
  		 \|\llbracket u_h\rrbracket \|_{0,e} \right)\nonumber\\[2mm]
 	&\lesssim &\left( \sum_{K \in \cT^h} \eta_{R_f,K}^2\right )^{1/2} 
 \!\!\!\tri E\tri_{\O}. 
\end{eqnarray}
To bound the second term in (\ref{error2}), first notice that 
 \[
	 \llbracket E-\cI_h E\rrbracket_e
 	=-\llbracket u_h+\cI_h E\rrbracket_e, 
	\quad \forall e \in \cE_I^h.
 \]
Since $u_h+\cI_h E \in \cU^{nc}$ and the the fact that $j_{\sigma,e}$ is a constant for all $e \in \cE^{h}$ , (\ref{orthogonality}) yields 
 \[
 	\int_e j_{\sigma,e}\, \llbracket E-E_h\rrbracket \,ds=0, 
	\quad \forall e \in \cE_I^h.
 \]
Hence, 
 \begin{equation} \label{global1.1}
 	\int_e  \{E-I_h E\}_e\,ds +\frac{1}{2}\int_e  
   	\llbracket E-I_h E\rrbracket_e\,ds 
	=\!\!\!\! \int_e(E-I_h E)|_e^+ \,ds
	=\!\!\!\!\int_e(E|_e^+-\pi_eE) \,ds,
\end{equation}
for all $e \in \cE_I^h$.
The last equality comes from the property of the nonconforming nodal basis functions:
$\fint_{e_i} \phi_{e_j}=\delta_{ij}.$
It then follows from (\ref{global1.1}), the Cauchy-Schwartz inequality, and Lemma \ref{lemma2} that
\begin{eqnarray} \label{global2}
 	I_2 
 	&=&\sum_{e\in \cE_I^h \cup \cE_N^h } \int_{e} j_{\sigma,e}\, (E|_e^+-\pi_e E) ds 
		 \le \sum_{e\in \cE_I^h \cup \cE_N^h } \| j_{\sigma,e} \|_{0,e} \, 
		 \big \|E|_e^+-\pi_{e} E \big\|_{0,e} \nonumber\\[2mm]
 	&\lesssim& \sum_{e\in \cE_I^h \cup \cE_N^h} 
 		\left(\frac{h_e}{\alpha^+_e}\right)^{1/2}  
 		 \| j_{\sigma,e}\|_{0,e} \, \tri E\tri_{0,K_e^+} 
		 \lesssim  \left(\sum_{K \in \cT^h}  \eta_{J_{\sigma},K}^2\right)^{1/2} \tri E\tri_{\O}.
\end{eqnarray}
Now, (\ref{global bound}) is a direct consequence of (\ref{error2}), (\ref{global1}),
(\ref{global2}), and Lemmas \ref{lemma3} and \ref{lemma4}. 
This completes the proof of the theorem.
\end{proof}

\section{Numerical Experiments}
\setcounter{equation}{0}

In this section, we report some numerical results for an interface
problem with intersecting interfaces used by many authors (see, e.g.,
\cite{Kel:74, CaYeZh:2011}), which is considered as a
benchmark test problem. Let $\O=(-1,1)^2$ and
 \[
 u(r,\theta)=r^{\beta}\mu(\theta)
 \]
in the polar coordinates at the origin with
$$
	\mu(\theta)=\left\{
	\begin{array}{lll}
	\cos((\pi/2-\sigma)\beta)\cdot\cos((\theta-\pi/2+\rho)\beta) &
	\mbox{if} & 0\leq \theta \leq \pi/2,\\[1ex]
	\cos(\rho\beta)\cdot\cos((\theta-\pi+\sigma)\beta) &
	\mbox{if} & \pi/2 \leq \theta \leq \pi,\\[1ex]
	\cos(\sigma\beta)\cdot\cos((\theta-\pi-\rho)\beta) &
	\mbox{if} & \pi\leq \theta \leq 3 \pi/2,\\[1ex]
	\cos((\pi/2-\rho)\beta)\cdot\cos((\theta-3\pi/2-\sigma)\beta) &
	\mbox{if} & 3\pi/2\leq \theta \leq 2\pi,
	\end{array}
\right.
$$
where $\sigma$ and $\rho$ are numbers. The function $u(r,\theta)$
satisfies the interface problem in (\ref{problem}) with
$\Gamma_N=\emptyset$, $f=0$, and
 \[
 	\a(x)=\left\{\begin{array}{ll}
 	R & \quad\mbox{in }\, (0,1)^2\cup (-1,0)^2,\\[2mm]
 	1 & \quad\mbox{in }\,\O\setminus ([0,1]^2\cup [-1, 0]^2).
	 \end{array}\right.
 \]
The numbers $\beta$, $R$, $\sigma$, and $\rho$ satisfy some
nonlinear relations. 
For example,
when $\beta=0.1$, then
 \[
 	R\approx 161.4476387975881, \quad \rho=\pi/4,
 	\quad\mbox{and}\quad
	 \sigma \approx -14.92256510455152.
 \]
Note that when $\beta=0.1$, this is a difficult problem for
computation.
\begin{remark}
This problem does not satisfy {\em Hypothesis 2.7} in {\em
\cite{BeVe:2000}} as the quasi-monotonicity is not satisfied about the origin.\end{remark}

Started with a coarse triangulation, a sequence of meshes is generated by using
a standard adaptive meshing algorithm that adopts the $L^2$ strategy: (i)
mark elements whose indicators are among the first 20 percent of 
the energy norm of the total error, and 
(ii) refine the marked triangles by bisection. The stopping criteria 
$$ \mbox{rel-err}:=\frac{\tri u-u_u\tri_\O}{\tri u\tri_\O} \le tol$$
is used, and numerical results with $tol=0.1$ are reported.
\begin{figure}[hb]
        \centering
    \begin{minipage}[htb]{0.48\linewidth}
        \includegraphics[trim=20 15 0 10,clip=true,width=1\textwidth,angle=0]{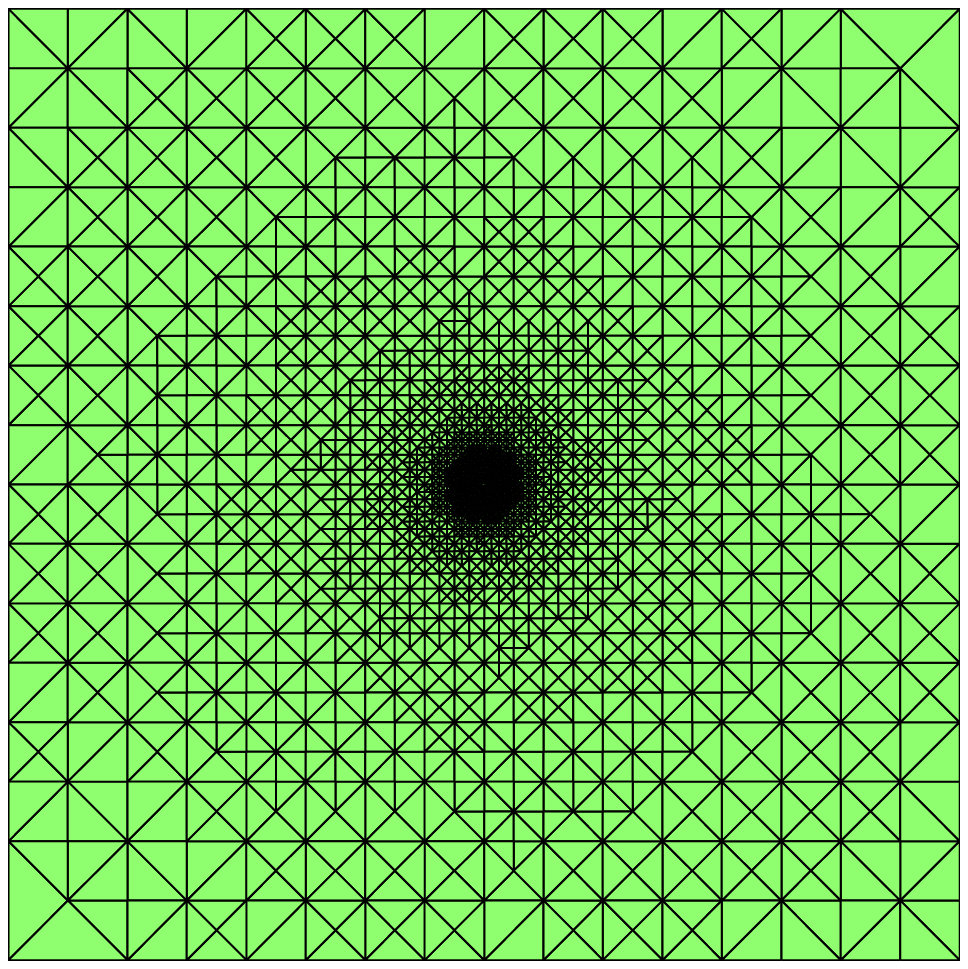}
        \vspace{-0.7cm}
        \caption{mesh generated by $\eta_K$.}
        \end{minipage}%
    \hspace{0.02\linewidth}
     \begin{minipage}[htb]{0.48\linewidth}
        \includegraphics[trim=20 15 0 10,clip=true,width=1\textwidth,angle=0]{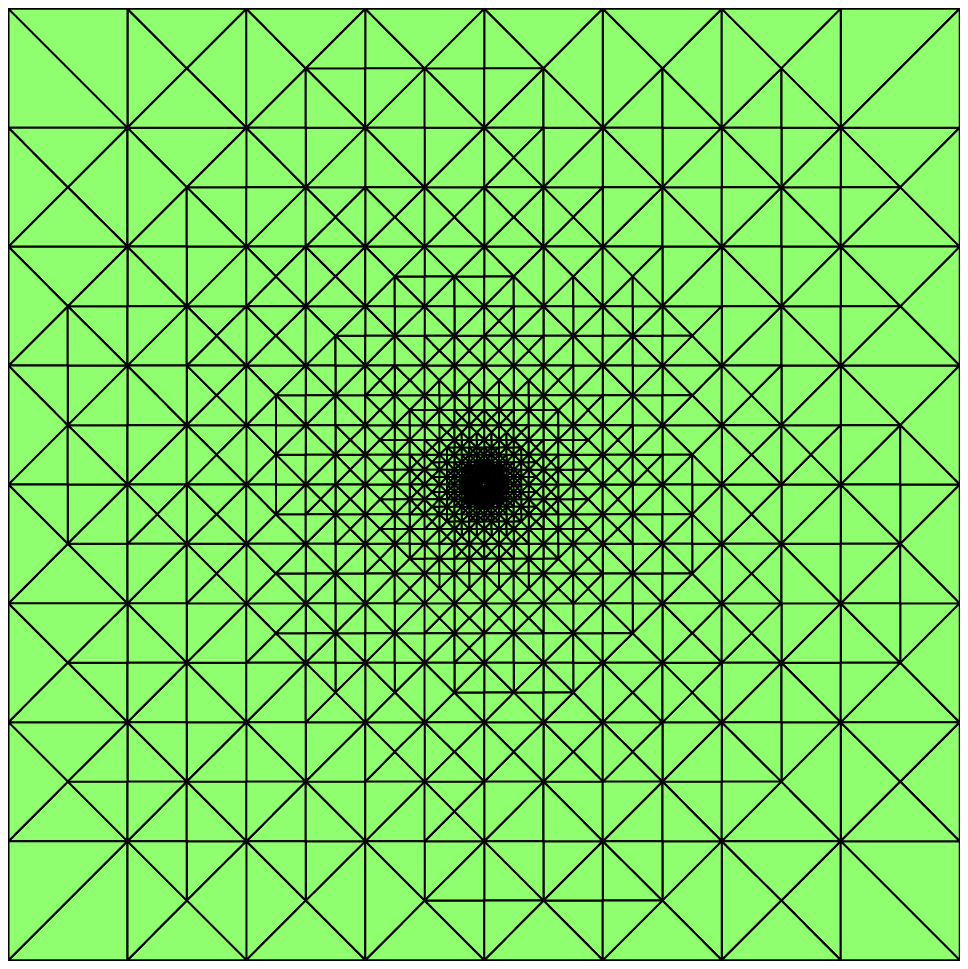}
        \vspace{-0.7cm}
        \caption{mesh generated by $\tilde \eta_K$.}
        \end{minipage}%
\end{figure}

\begin{figure}[ht]
        \centering
    \begin{minipage}[htb]{0.48\linewidth}
        \includegraphics[trim=20 15 0 10,clip=true,width=1\textwidth,angle=0]{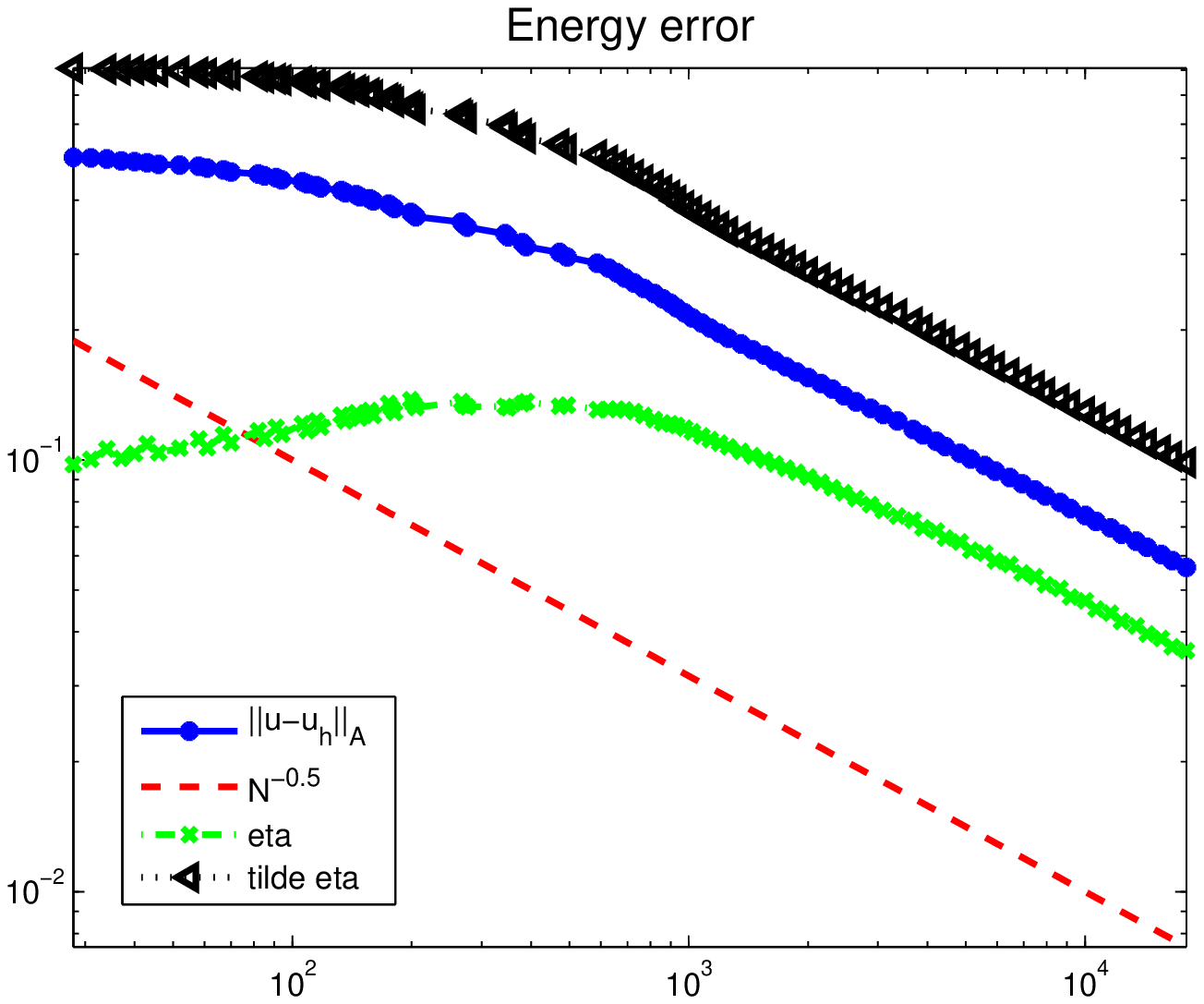}
        \vspace{-0.7cm}
        \caption{error and estimators on mesh generated by $\eta_K$.}
        \end{minipage}%
    \hspace{0.02\linewidth}
     \begin{minipage}[htb]{0.48\linewidth}
        \includegraphics[trim=20 15 0 10,clip=true,width=1\textwidth,angle=0]{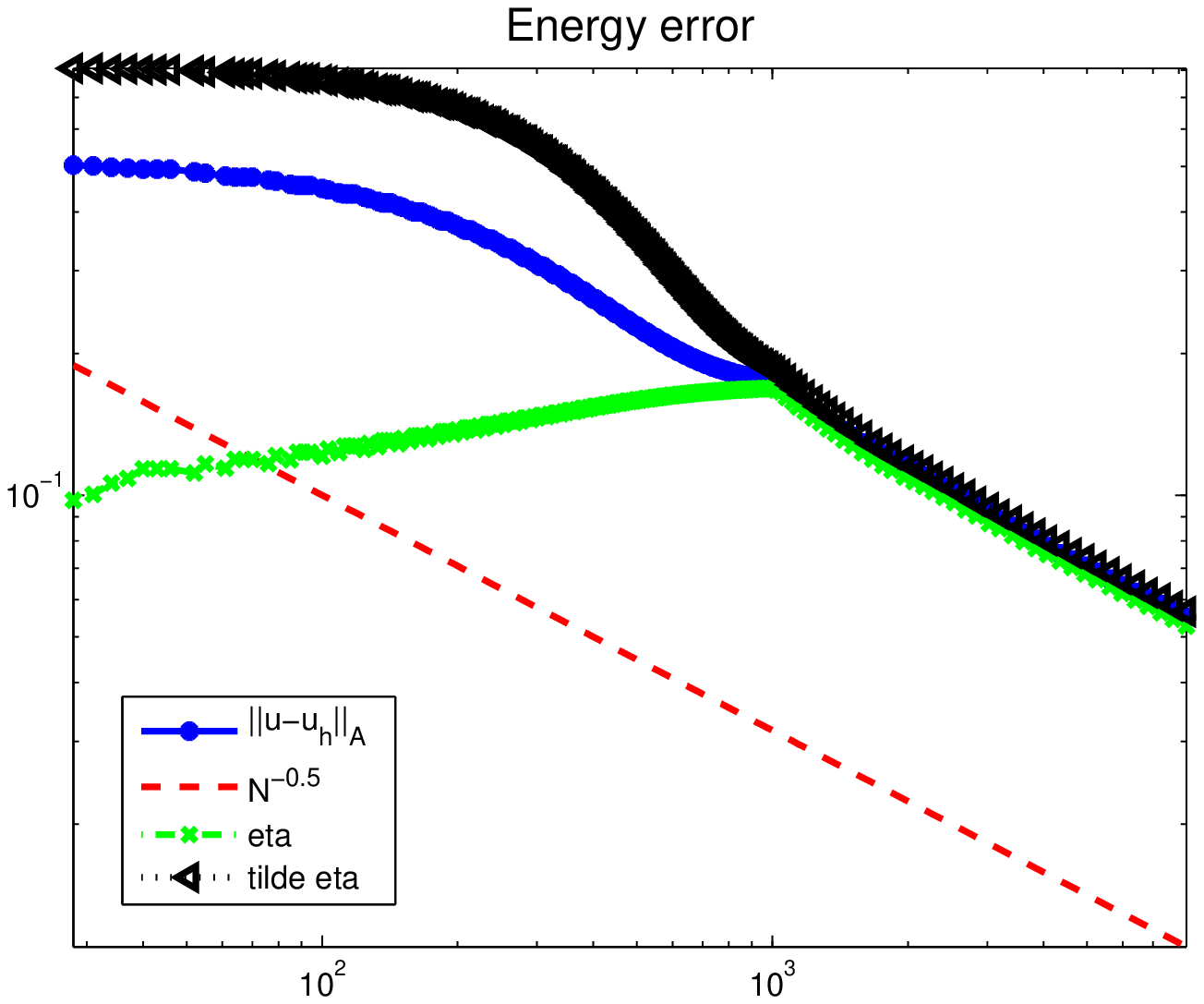}
        \vspace{-0.7cm}
        \caption{error and estimators on mesh generated by $\tilde{\eta}_K$.}
        \end{minipage}%
\end{figure}

Meshes generated by the standard and modified indicators, $\eta_K$ and $\tilde \eta_K$,
are depicted respectively in Figures $1$ and $2$. Both the refinements are centered 
at the origin. There are 11974 and 5524 elements
in the respective
Figures 1 and 2. Hence,
this test problem suggests that the modified indicator generates a much
better mesh than the standard indicator even though the local efficiency bound
of the modified indicator depends on the jump of the diffusion coefficient.

The comparisons of the true error and the estimators $\eta$ and $\tilde{\eta}$ on the meshes
generated by the standard and modified indicators are shown in Figures 3 and 4, respectively. 
The slope of the log(dof)- log(error) for both the estimators on both the meshes are very close to $-1/2$, which indicates the optimal decay of the error with respect to the number of unknowns.
The efficiency index is defined by
 \[\mbox{eff-index}:=\dfrac{\mbox{estimator}}{\tri u-u_h\tri_\O}.
\]
The efficinecy indices for the $\eta$ and $\tilde{\eta}$
are about $0.6404$ and $1.7469$ on the mesh generated by $\eta_K$
and about $0.9582$ and  $1.0282$ on the mesh generated by $\tilde{\eta}_K$, respectively. 

 \begin{figure}[h]
        \centering
         \begin{minipage}[!hbp]{0.48\linewidth}
        \centering
        \includegraphics[width=0.99\textwidth,angle=0]{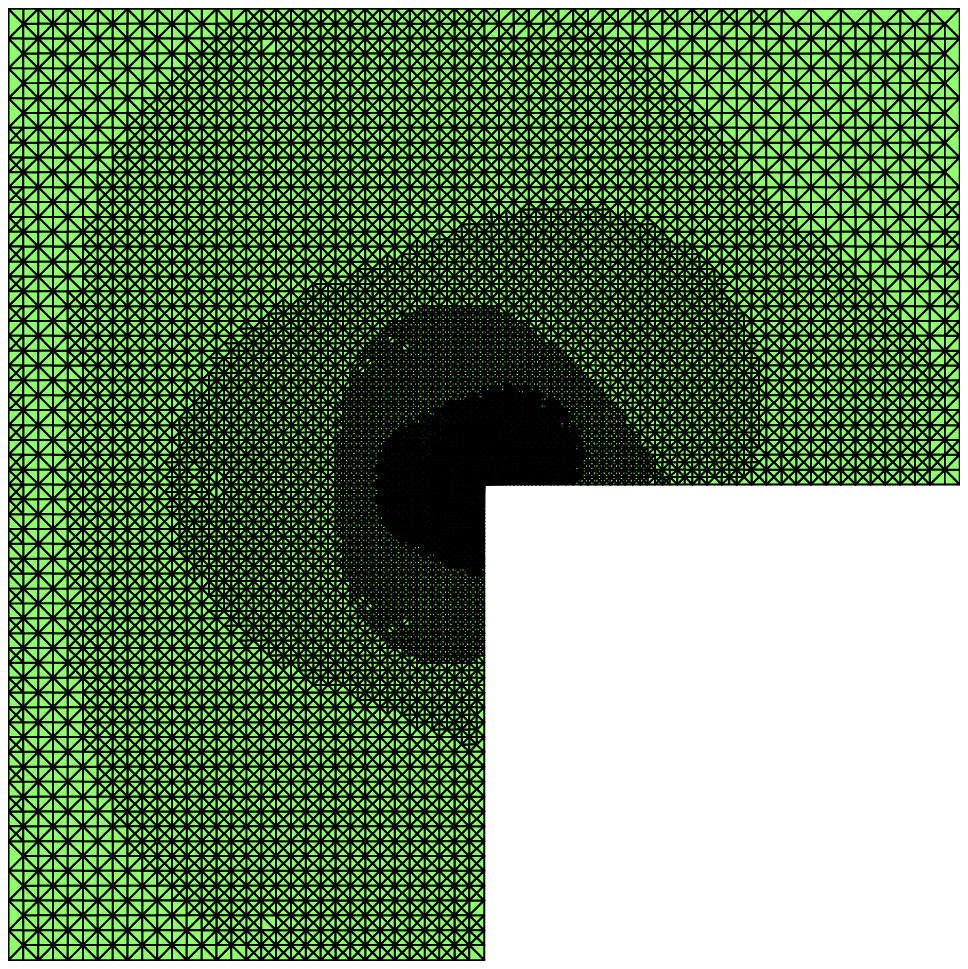}
        \caption{mesh generated by $\eta_K$ with solution jump}%
        \end{minipage}%
        \quad
    \begin{minipage}[!htbp]{0.48\linewidth}
        \centering
        \includegraphics[width=0.99\textwidth,angle=0]{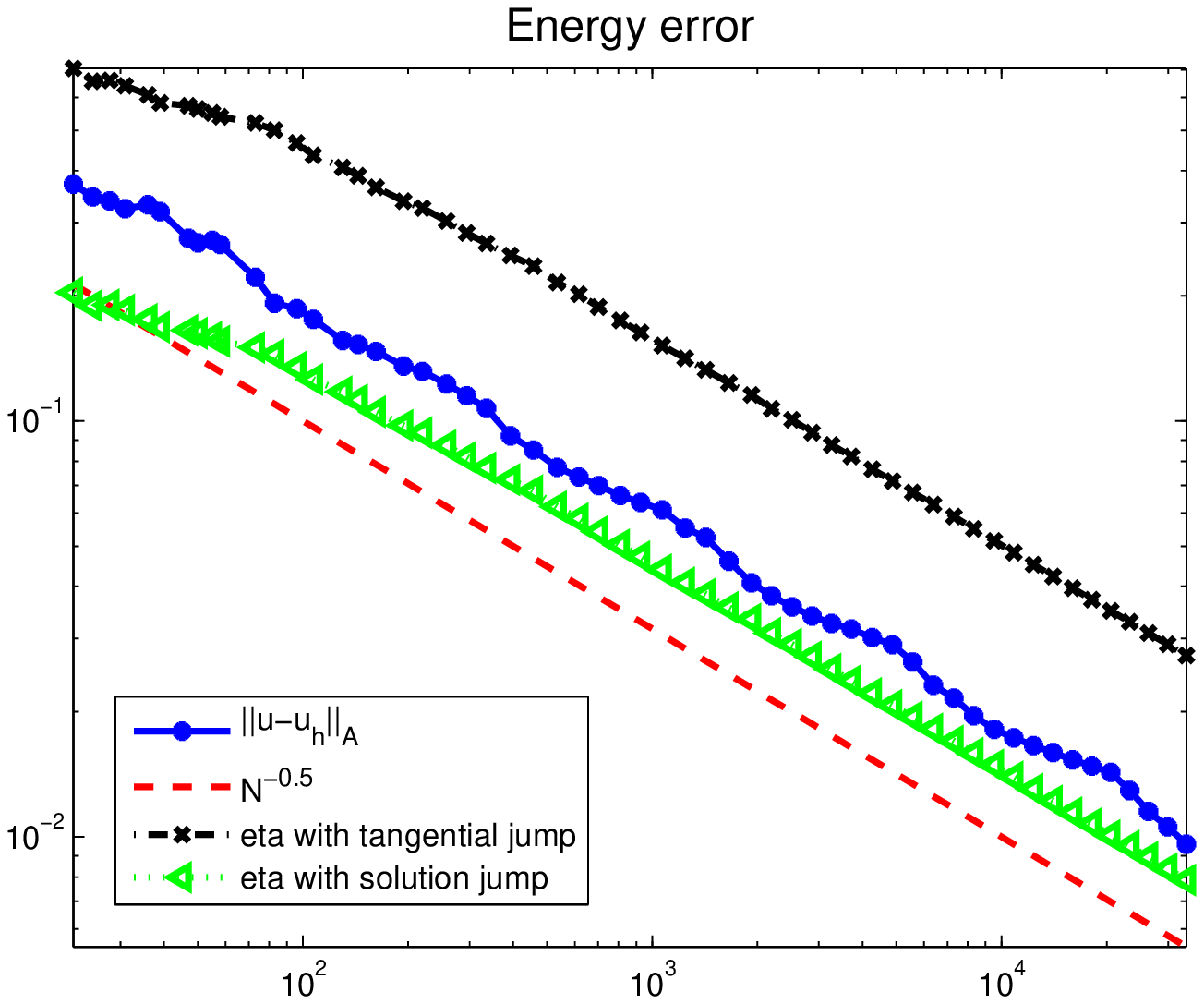}
        \caption{error and estimators on mesh generated by $\eta_K$}%
    \end{minipage}%
        \hfill
\end{figure}
Our numerical results also show that our standard estimator using the edge solution jump 
is more accurate than
the existing estimators using the edge tangential jump. To illustrate this fact, 
we present numerical results for a test problem \cite{EiSiWa:2005}:
a Poisson equation defined on the L-shaped domain $\Omega=(-1,\,1)^2\setminus [0,\,1]\times [-1,\,0]$ with 
the following exact solution 
\[
u(x,\,\theta)=r^{2/3}\sin \left(\dfrac{2\theta+\pi}{3} \right), \quad \theta \in [0, \, 3\pi/2].
\]

The stopping criteria is set as $tol \le 0.0075$. The efficiency indices in the final step 
are  $0.8205$ and $2.8423$ for the respective estimators with the edge solution and tangential
derivative jumps. 
This indicates that our standard estimator is more accurate than the existing 
residual estimator (see Figure 6).

\bibliographystyle{amsplain}

\begin{thebibliography}{10}
\bibitem{Ai:2005}
{\sc M. Ainsworth},
{\em Robust a posteriori error estimation for nonconforming finite element approximation},
SIAM J. Numer. Anal., 42:6 (2005), 2320-2341.

\bibitem{Ai:2007}
{\sc M. Ainsworth},
{\em A posteriori error estimation for discontinuous Galerkin finite element approximation},
SIAM J. Numer. Anal., 45:4 (2007), 1777-1798.

\bibitem{BeHaLa:2003}
{\sc P. Becker, P. Hansbo, and M. G. Larson},
{\em Energy norm a posteriori error estimation for discontinuous Galerkin method},
Comput. Meth. Appl. Mech. Engrg., 192 (2003), 723-733.

\bibitem{BeVe:2000}
{\sc C. Bernardi and R. Verf\"urth},
{\em Adaptive finite element methods for elliptic equations with non-smooth coefficients},
Numer. Math., 85 (2000), 579-608.

\bibitem{CaBaJa:2002}
{\sc C. Carstensen, S. Bartels, and S. Jansche},
{\em A posteriori error estimates for nonconforming finite element methods},
Numer. Math., 92 (2002), 233-256.

\bibitem{CaHuOr:2007}
{\sc C. Carstensen, J. Hu, and A. Orlando},
{\em Framework for the a posteriori error analysis of the nonconforming finite elements},
SIAM J. Numer. Anal., 45:1 (2007), 68-82.

\bibitem{CaYeZh:2011}
{\sc Z. Cai, X. Ye, and S. Zhang},
{\em Discontinuous Galerkin finite element methods for interface problems: a priori and a posteriori error estimations},
SIAM J. Numer. Anal., 49:5 (2011), 1761-1787.

 \bibitem{CaZh:09}
 {\sc Z. Cai and S. Zhang},
 {\em Recovery-based error estimator for interface problems:
 conforming linear elements}, SIAM J. Numer. Anal., 47:3 (2009), 2132-2156.
 
\bibitem{CaZh:2010}
{\sc Z. Cai and S. Zhang},
{\em Flux recovery and a posteriori error estimators: conforming elements for scalar elliptic equations},
SIAM J. Numer. Anal., 48:2 (2010), 578-602.

 \bibitem{CaZh:10a}
 {\sc Z. Cai and S. Zhang},
 {\em Recovery-based error estimator for interface problems: mixed and nonconforming elements},
  SIAM J. Numer. Anal., 48:1 (2010), 30-52.


\bibitem{CaZh:10c}
 {\sc Z. Cai and S. Zhang},
{\em Robust residual- and recovery a posteriori error estimators for interface problems with
flux jumps},  Numer. Methods for PDEs, 28:2 (2012), 476-491

 \bibitem{CaZh:11}
 {\sc Z. Cai and S. Zhang},
{\em Robust equilibrated residual error estimator for diffusion problems: conforming elements},
SIAM J. Numer. Anal., 50:1 (2012), 151-170.




\bibitem{DaDuPaVa:1995}
{\sc E. Dari, R. Duran, and C. Padra},
{\em Error estimations for nonconforming finite element approximations of the Stokes problem},
Math. Comp., 64 (1995), 1017-1033.


\bibitem{DaDuPaVa:1996}
{\sc E. Dari, R. Duran, C. Padra, and V. Vampa},
{\em A posteriori error estimators for nonconforming finite element methods},
RAIRO Mod\'el Anal. Num\'er., 30:4 (1996), 385-400.

\bibitem{DrSaWi:1996}
{\sc M. Dryja, M. V. Sarkis, and O. B. Widlund},
{\em Multilevel Schwartz method for elliptic problems with discontinuous in three dimensions},
Numer. Math., 72 (1996), 313-348.

\bibitem{EiSiWa:2005}
{\sc H. Elman, D. Silvester, and A. Wathen},
{\em Finite Elements and Fast Iterative Solvers: With
Applications in Incompressible Fluid Dynamics},
Numer. Math. Sci. Comput., Oxford University Press, Oxford, UK, 2005.

\bibitem{GiRa:86}
{\sc V. Girault and P. A. Raviart},
{\em Finite Element Methods for Navier-Stokes Equations},
Springer-Verlag, Berlin, 1986.

\bibitem{HoWo:1996}
{\sc R. H. W. Hoppe and B. Wohlmuth},
{\em Element-oriented and edge-oriented local error estimators for nonconforming finite element methods},
RAIRO Mod\'el. Math. Anal. Num\'er., 30:2 (1996), 237-263.

 \bibitem{Kel:74}
 {\sc R. B. Kellogg},
 {\em On the Poisson equation with intersecting interfaces},
 Appl. Anal., 4 (1975), 101-129.
 
 \bibitem{Kim:07}
 {\sc K. Y. Kim},
 {\em A posteriori error analysis for locally conservative mixed methods},
 Math. Comp., 76 (2007), 43-66.

 \bibitem{LoSt:2006}
 {\sc C. Lovadina and R. Stenberg},
 {\em Energy norm a posteriori error estimates for mixed finite element methods},
 Math. Comp., 75 (2006), 1659-1674.
 
 \bibitem{LuWo:04} 
 {\sc R. Luce and B. I. Wohlmuth}, 
 {\em A local a posteriori error estimator based on equilibrated fluxes}, 
 SIAM J. Numer. Anal., 42:4 (2004), 1394-1414.

\bibitem{Pe:2002}
{\sc M. Petzoldt},
{\em A posteriori error estimators for elliptic equations with discontinuous coefficients},
Adv. Comput. Math., 16 (2002), 47-75.

\bibitem{Sc:2002}
{\sc F. Schieweck},
{\em A posteriori error estimates with post-processing for nonconforming finite elements},
ESAIM Math. Mod. Numer. Anal., 36:3 (2002), 489-503.


\bibitem{Ve:1996}
{\sc R. Verf\"urth},
{\em A Posteriori Error Estimation Techniques for Finite Element Methods},
Oxford University Press, Oxford, United Kingdom, 2013.

\bibitem{Voh:11}
{\sc M. Vohral\'ik},
{\em Guaranteed and fully robust a posteriori error estimates for conforming discretizations of diffusion problems with discontinuous coefficients}, 
J. Sci. Comput., 46:3 (2011), 397-438

\end{thebibliography}

\end{document}